\documentclass[11pt]{amsart}
\usepackage[leqno]{amsmath}
\usepackage{amsthm}
\usepackage{amsfonts}
\usepackage{amssymb}
\usepackage{eucal}
\usepackage{mathrsfs}
\usepackage[all]{xy}

\theoremstyle{plain}
\newtheorem{thm}{Theorem}[section]
\newtheorem{prop}[thm]{Proposition}
\newtheorem{lemma}[thm]{Lemma}
\newtheorem{cor}[thm]{Corollary}

\newtheorem{pro}[thm]{Problem}

\theoremstyle{definition}
\newtheorem{ex}[thm]{Example}
\newtheorem{defi}[thm]{Definition}
\newtheorem{rem}[thm]{Remark}

\newcommand{\pa}[2]{\ensuremath{\stackrel{#1}{#2}}}
\newcommand{\mc}[1]{\ensuremath{\mathcal{#1}}}

\newcommand{\mb}[1]{\ensuremath{\mathbb{#1}}}
\newcommand{\ra}{\rightarrow}

\newcommand{\xra}{\xrightarrow}

\newcommand{\mono}{\hookrightarrow}

\newcommand{\ol}[1]{\ensuremath{\overline{#1}}}
\newcommand{\wh}[1]{\ensuremath{\widehat{#1}}}

\newcommand{\ul}[1]{\ensuremath{\underline{#1}}}
\newcommand{\pgl}[1]{\ensuremath{\mb{P}GL(#1)}}

\newcommand{\mker}{\ensuremath{{\rm Ker}}}

\newcommand{\mdim}{\ensuremath{{\rm dim}}}
\newcommand{\stab}{\ensuremath{{\rm Stab}}}
\newcommand{\PP}{\mathbb{P}}

\newcommand{\CC}{\mathbb{C}}

\begin{document}

\title{Dense $\mb{P}GL$-orbits in products of Grassmannians}
\author[I. Coskun]{Izzet Coskun}
\author[M. Hadian]{Majid Hadian}
\address{University of Illinois at Chicago, Department of Mathematics, Stat. \& CS, 851 S Morgan St, Chicago IL 60607}
\email{coskun@math.uic.edu}
\email{hadian@math.uic.edu}
\author[D. Zakharov]{Dmitry Zakharov}
\address{Courant Institute of Mathematical Sciences, 251 Mercer Street, New York, NY 10012}
\email{dvzakharov@gmail.com}
\date{}
\thanks{During the preparation of this article the first author  was partially supported by the NSF CAREER grant DMS-0950951535, and an Alfred P. Sloan Foundation Fellowship}

\maketitle


\begin{abstract}

In this paper, we find some necessary and sufficient conditions on the dimension vector $\ul{\bf{d}} = (d_1, \dots, d_k; n)$ so that the diagonal action of $\pgl{n}$ on $\prod_{i=1}^k Gr(d_i;n)$ has a dense orbit. Consequently, we obtain some algorithms for finding dense and sparse dimension vectors and classify dense dimension vectors with small length or size. We also classify dimension vectors where $|d_i - d_j|< 3$ for all $i,j$ generalizing a theorem of Popov \cite{p}.

\end{abstract}

\tableofcontents


\section{Introduction}\label{sec-intro}

Let $\{ p_1, p_2, p_3 \}$ and $\{ q_1, q_2, q_3 \}$ be two ordered sets of distinct points on $\mb{P}^1$. Then there exists a unique M\"obius transformation $M \in \pgl{2}$ such that $M(p_i) = q_i$ for $1 \leq i \leq 3$. Hence, the diagonal action of $\pgl{2}$ on $\mb{P}^1 \times \mb{P}^1 \times \mb{P}^1$ has a dense orbit consisting of distinct triples. More generally, given two ordered sets $\{ p_1, \dots, p_{n+2} \}$ and $\{ q_1, \dots, q_{n+2} \}$ of $n+2$ points in general linear position in $\mb{P}^n$, there exists a unique element $M \in \pgl{n+1}$ such that $M(p_i) = q_i$ for $1 \leq i \leq n+2$ (see \cite[Section 1.6]{har92}). In this paper, we consider a natural generalization of this classical fact.

Let $V$ be an $n$-dimensional vector space. Then any ordered set $S = \{ U_1, \dots, U_k \}$ of linear subspaces of $V$ corresponds to a point in $\prod_{i=1}^k Gr(d_i; n)$, where $d_i$ denotes the dimension of $U_i$ for $1 \leq i \leq k$. Therefore, the $GL(n)$-action on $V$ induces a $\pgl{n}$ action on $\prod_{i=1}^k Gr(d_i;n)$. 

\begin{pro}\label{q:main}
For which dimension vectors $(d_1, \dots, d_k; n)$ does the diagonal action of $\pgl{n}$ have a dense orbit in $\prod_{i=1}^k Gr(d_i;n)$?
\end{pro}
Problem \ref{q:main} was posed to us by J\'anos Koll\'ar. This problem and several variations have a long history.  Let $X$ be an irreducible algebraic variety admitting a non-trivial action of a connected algebraic group $G$. The density of the diagonal action of $G$ on $X^m$  has been studied by Vladimir Popov in \cite{p}.  In particular, he characterized the values of $m$ with a dense orbit for the standard action of $G$ on $X=G/P$, where $P$ is a maximal parabolic subgroup and $G$ is a simple linear algebraic group.  Popov in \cite{p2} studied the relation of the problem to representation theory. Many authors have studied variants of the problem. For instance,  Magyar, Weyman and Zelevinsky have classified instances when the action of $G$ on the product has finitely many orbits \cite{MWZ1} (see also \cite{MWZ2} and \cite{perrin}, \cite{kimelfeld}, \cite{kv}, \cite{ponomareva} for related literature).

In this paper, we prove reduction lemmas that allow to reduce the density of a dimension vector to one with smaller ambient dimension under suitable assumptions. We work over an algebraically closed field of arbitrary characteristic. For many dimension vectors, our results give an efficient algorithm for checking the density of the $\pgl{n}$ action (see Section \ref{s:fq}).
Using our results, we can  characterize the dimension vectors of length $k \leq 4$ that have a dense $\pgl{n}$ orbit (Theorem \ref{t:sl}) and the dimension vectors with $d_i \leq 4$ (for all $i$) that have a dense $\pgl{n}$ orbit (Section \ref{sec-size}). We also characterize dense dimension vectors where $|d_i - d_j|<3$ (Theorem \ref{t:ed}). The latter generalizes  Popov's  \cite[Theorem 3]{p} in type $A$.

If $\pgl{n}$ acts with a dense orbit on $\prod_{i=1}^k Gr(d_i;n)$, then the dimension of $\pgl{n}$ has to be greater than or equal to the dimension of $\prod_{i=1}^k Gr(d_i;n)$. We thus obtain a necessary inequality
\begin{equation} \label{e:tdc}
\sum_{i=1}^k d_i (n-d_i) \leq n^2 -1.
\end{equation}
However, as the following example shows, this inequality is not sufficient.

\begin{ex}\label{ex-1}
Consider the dimension vector $(1,1,2,2;3)$. Geometrically, this dimension vector represents a configuration $(p_1,p_2,l_1,l_2)$ consisting of a pair of points $(p_1,p_2)$ and a pair of lines $(l_1,l_2)$ in $\mb{P}^2$. Note that we have
$$\mdim(\mb{P}^2 \times \mb{P}^2 \times \mb{P}^{2*} \times \mb{P}^{2*}) = 8 = \mdim(\pgl{3}).$$
However, $\pgl{3}$ does not act with a dense orbit. Briefly, the points $p_1$ and $p_2$ span a line $l$ in $\mb{P}^2$ and the lines $l_1$ and $l_2$ intersect $l$ in two points $q_1$ and $q_2$. The cross-ratio of the four points $p_1$, $p_2$, $q_1$, and $q_2$ on $l$ is an invariant of the $\pgl{3}$ action. Furthermore, by fixing $p_1$, $p_2$, and $l_1$ and varying $l_2$, we can get every cross-ratio. Hence, all orbits of the $\pgl{3}$-action in this case have codimension at least $1$.
\end{ex}

There are several things to notice about Example \ref{ex-1}. First, it can be generalized to the dimension vector $(1,1,n-1,n-1;n)$. Geometrically, this dimension vector represents a configuration $(p_1, p_2, H_1, H_2)$ consisting of a pair of points $(p_1, p_2)$ and a pair of hyperplanes $(H_1, H_2)$ in $\mb{P}^{n-1}$. Again, the hyperplanes $H_1$ and $H_2$ intersect the line $l$ spanned by the points $p_1$ and $p_2$ in two points $q_1$ and $q_2$ and the cross-ratio of the four points $p_1$, $p_2$, $q_1$, and $q_2$ on $l$ is an invariant of the $\pgl{n}$ action. Therefore, $\pgl{n}$ does not act with a dense orbit in this case. On the other hand, $\mdim(\mb{P}^{n-1} \times \mb{P}^{n-1} \times \mb{P}^{n-1*} \times \mb{P}^{n-1*}) = 4(n-1)$ can be arbitrarily smaller than $\mdim(\pgl{n}) = n^2 -1$.

More importantly,  in all the above examples there is a smaller configuration of linear spaces (the points $p_1$, $p_2$, $q_1$, and $q_2$ on $l$) obtained by taking spans and intersections of the original linear spaces, that trivially cannot be dense, as it fails the inequality \eqref{e:tdc}.  This smaller configuration of linear spaces is the obstruction for the density of the original dimension vector. In this paper, we will produce many classes of examples which show that this phenomenon is typical. In fact, we expect that whenever $\pgl{n}$ fails to act with a dense orbit on $\prod_{i=1}^k Gr(d_i;n)$, there is a configuration of vector spaces, obtained by repeatedly taking spans and intersections of the original ones, which does not satisfy the inequality \eqref{e:tdc} and accounts for this failure.

Knowing the density of a dimension vector has many applications. For the applications we specialize our base field to $\CC$. Let $\lambda_1, \dots, \lambda_d$ be nonzero dominant characters of the maximal torus $T$ in the semisimple group $G$. Then $(\lambda_1, \dots, \lambda_d)$ is called {\em primitive} if for every non-negative $d$-tuple of integers $(n_1, \dots, n_d)$, the Littlewood-Richardson coefficient $c_{n_1 \lambda_1, \dots, n_d \lambda_d}^0\leq 1$. Popov in \cite[Theorem 1]{p2} proves that if $G$ has an open orbit on $G/P_{\lambda_1} \times \cdots \times G/P_{\lambda_d}$, then $(\lambda_1, \dots, \lambda_d)$ is primitive. Hence, for each vector that we prove that $\pgl{n}$ acts with dense orbit, we get a strong bound on Littlewood-Richardson coefficients.

Knowing the density also has important geometric applications. First, it allows one to choose convenient coordinates. Many geometric problems, such as enumerative problems and interpolation problems, become simpler to solve if the constraints have special coordinates. For example, it is easy to see that the unique quadric surface in $\PP^3$ that contains the three lines $x=y=0, z=w=0, x-z=y-w=0$ is $xw-yz=0$. Since $(2,2,2; 4)$ is dense and these three lines belong to the dense orbit (see proof of Theorem \ref{t:sl}), we conclude that  three general lines in $\PP^3$ impose independent conditions on quadrics in $\PP^3$ and there is a unique quadric surface containing  three general lines (see \cite{har92}). More generally,  $(k,k,k; 2k)$ is dense and $x_1 = \cdots = x_{k} = 0$, $x_{k+1} = \cdots = x_{2k}=0$ and $x_1-x_{k+1} = \cdots = x_k - x_{2k} =0$ is in the dense orbit of $\pgl{2k}$. It is then easy to see that there is a unique Segre image of $\PP^1 \times \PP^{k-1}$ containing these three $\PP^{k-1}$'s in $\PP^{2k-1}$ (see \cite{har92}). This simple classical calculation is the basis for the study of the genus zero Gromov-Witten invariants of Grassmannians (see \cite{coskun:LR}).

Second, knowing the density of  dimension vectors  allows one to determine automorphism groups of blowups of Grassmannians.  For example, the dimension vector $(2,2,2,2;5)$ is dense (see Theorem \ref{t:sl}). Hence,  any 4 general points in $G(2,5)$  can be taken to any other 4 general points by an action of $\pgl{5}$. By Lemma \ref{l:ml}, if a 4-tuple of points is in the dense orbit of $\pgl{5}$, then any ordering of the 4 points is also in the dense orbit. Hence, there is an $\mathfrak{S}_4$-symmetry of the blowup of $G(2,5)$ at 4 general points. More interestingly, the blowup $X$ of $G(2,5)$ at a general $\PP^3$ section under the Pl\"{u}cker embedding has an $\mathfrak{S}_5$-symmetry (arising from the $\mathfrak{S}_4$-symmetry) that plays an important role in the Kawamata-Morrison cone conjecture for the log Calabi-Yau variety $X$ (see \cite{coskun:artie}).

\smallskip

\noindent{\bf Organization of the paper:} In \S \ref{sec-setup}, we will prove a useful criterion for checking density in terms of dimensions of stabilizer groups. In \S \ref{sec-numerical}, we will collect some numerical lemmas. In \S \ref{s:rt}, we will prove several lemmas that allow us to reduce checking the density of a dimension vector to  simpler dimension vectors. In \S \ref{sec-length}, we will characterize dense dimension vectors of length at most four. In \S \ref{sec-size}, we will characterize dense dimension vectors of size at most four. In \S \ref{s:ed}, we characterize the dense vectors with $|d_i - d_j| < 3$. Finally, in \S \ref{s:fq}, we  discuss several additional examples and further questions.
\smallskip

\noindent{\bf Acknowledgments:} We would like to thank J\'{a}nos Koll\'{a}r for bringing the question to our attention. We are grateful to Clay Cordova, Samuel Grushevsky, Joe Harris, Brendan Hassett, J\'{a}nos Koll\'{a}r, Vladimir L. Popov and Ravi Vakil for many enlightening conversations. We thank the referee for invaluable suggestions and pointing us to many  useful references.


\section{Setup and the main lemma}\label{sec-setup}

In this section, we prove a basic criterion for the density of the $\pgl{n}$ action in terms of stabilizers.

Let $(U_1, \dots, U_k)$ be a configuration of linear subspaces of an $n$-dimensional vector space $V$. The corresponding dimension vector will be denoted by $\ul{\bf{d}} = (d_1, \dots, d_k; n)$, where $d_i$ is the dimension of $U_i$ for all $i$. The number of subspaces $k$ is called the {\em length} of the dimension vector $\ul{\bf{d}}$ and will be denoted by $l(\ul{\bf{d}})$. The value $\max_i{d_i}$ is called the {\em size} of $\ul{\bf{d}}$ and will be denoted by $| \ul{\bf{d}} |$.  The number $n$ is the {\em ambient dimension}, while the sum $\sum_{i=1}^k d_i$ is the {\it total dimension} and the difference $\sum_{i=1}^k d_i-n$ is the {\it excess dimension}. When convenient, we will express a dimension vector in exponential notation $\ul{\bf{d}} = (1^{e_1}, \dots, (n-1)^{e_{n-1}};n)$. Finally, the stabilizer of the point $(U_1, \dots, U_k) \in \prod_{i=1}^k Gr(d_i;n)$ will be denoted by $\stab(U_1, \dots, U_k)$.

\begin{defi}
We say that a dimension vector $\ul{\bf{d}} = (d_1, \dots, d_k;n)$ is  {\em dense} if the diagonal action of $\pgl{n}$ on $\prod_{i=1}^k Gr(d_i;n)$ has a Zariski dense orbit. Otherwise, $\ul{\bf{d}}$ is {\em sparse}. Furthermore, we say $\ul{\bf{d}}$ is {\em trivially sparse} if it does not satisfy the dimension inequality \eqref{e:tdc}.
\end{defi}

The following basic lemma is going to be the main tool in this article.

\begin{lemma} \label{l:ml}
Let $(U_1, \dots, U_k) \in \prod_{i=1}^k Gr(d_i;n)$ be a tuple of vector spaces with dimension vector $\ul{\bf{d}} = (d_1, \dots, d_k;n)$. The $\pgl{n}$ orbit of this point is dense in $\prod_{i=1}^k Gr(d_i;n)$ if and only if
$$\mdim(\stab(U_1, \dots, U_k)) = (n^2 - 1) - \sum_{i=1}^k d_i(n-d_i).$$
\end{lemma}

\begin{proof}
If an algebraic group $G$ acts on an irreducible projective variety $X$,  the orbit $G.x$ of any point $x\in X$ under $G$ is open in its Zariski closure $\ol{G.x}$ \cite[I.1.8]{bor91}. On the other hand, the orbit $G.x$ is in bijection with $G/H$, where $H$ is the stabilizer of $x$. Consequently, the dimension of the Zariski closure of the orbit of $x$ is
$$\mdim(\ol{G.x}) = \mdim(G) - \mdim(H).$$
Hence, the orbit $G.x$ is dense in $X$ if and only if $\mdim(X) = \mdim(G) - \mdim(H)$. Specializing to the case $G = \pgl{n}$ and $X = \prod_{i=1}^k Gr(d_i;n)$, we obtain the lemma since $\mdim(\pgl{n}) = n^2 -1$ and
$$\mdim\left(\prod_{i=1}^k Gr(d_i;n)\right) = \sum_{i=1}^k d_i(n-d_i).$$
\end{proof}

\begin{rem}
Observe that $\mdim(\stab(U_1, \dots, U_k)) \geq (n^2 - 1) - \sum_{i=1}^k d_i(n-d_i)$. A dimension vector $\ul{\bf{d}}$ is dense if there is a $k$-tuple $(U_1, \dots, U_k) \in \prod_{i=1}^k G(d_i;n)$ where equality is achieved.
\end{rem}

Let us conclude this section with the following two evident but useful observations.

\begin{defi}
Let $\ul{\bf{d}} = (d_1, \dots, d_k; n)$ be a dimension vector. Then the {\em complement} of \ul{\bf{d}} is the dimension vector $\ul{\bf{d}}^c = (n-d_1, \dots, n-d_k;n)$.
\end{defi}

The following lemma is well-known (see \cite[Lemma 2 and Corollary 1, 2]{E}  and \cite[\S2]{ST}). We include the  proof for the reader's convenience.

\begin{lemma} \label{l:c}
A dimension vector $\ul{\bf{d}}$ is dense if and only if the complement dimension vector $\ul{\bf{d}}^c$ is dense.
\end{lemma}

\begin{proof}
Consider the ambient vector space $V$ and its dual $V^*$ with the dual $\pgl{n}$ action. Taking quotient spaces and passing to the dual defines an isomorphism $\prod_{i=1}^k Gr(d_i;n)$ and $\prod_{i=1}^k Gr(n-d_i;n)$, which respects the $\pgl{n}$ action. Therefore, $\pgl{n}$ has a dense orbit on one if and only if  it does on the other.
\end{proof}

\begin{defi}
We say that a dimension vector $\ul{\bf{d}} = (d_1, \dots, d_k;n)$ {\em dominates} a dimension vector $\ul{\bf{d}}' = (d'_1, \dots, d'_{k'};n)$ if for every $1 \leq d \leq n-1$ the number of times that $d$ appears in $\ul{\bf{d}}$ is greater than or equal to the number of times it appears in $\ul{\bf{d}}'$.
\end{defi}

The following is immediate.

\begin{lemma}\label{2.7}
Let $\ul{\bf{d}}$ and $\ul{\bf{d}}'$ be two dimension vectors such that $\ul{\bf{d}}$ dominates $\ul{\bf{d}}'$. Then if $\ul{\bf{d}}'$ is sparse, so is $\ul{\bf{d}}$. Equivalently, if $\ul{\bf{d}}$ is dense, so is $\ul{\bf{d}}'$.
\end{lemma}


\section{A numerical lemma}\label{sec-numerical}

In this section, we prove a numerical lemma which will be very useful in the next section in reducing the density problem of a dimension vector to a smaller one.

\begin{lemma} \label{l:nl}
Let $\ul{\bf{d}} = (d_1, \dots, d_k, d_{k+1}, d_{k+2}; n)$ be a dimension vector listed in increasing order $d_1 \leq d_2 \leq \cdots \leq d_{k+2}$ and with $| \ul{\bf{d}} | \leq \frac{n}{2}$. Then either \ul{\bf{d}} is trivially sparse or $\sum_{i=1}^k d_i < n$.
\end{lemma}

\begin{proof}
First, if all $d_i$ are  between $1$ and $\frac{n}{2}$, then $d_i(n-d_i)$ is a strictly increasing function of $d_i$. Hence, by decreasing the $d_i$'s if necessary, it suffices to show that if $\sum_{i=1}^k d_i = n$, then $\sum_{i=1}^{k+2} d_i(n-d_i) \geq n^2$. Second, by decreasing $d_{k+1}$ and $d_{k+2}$ if necessary, we can assume that $d_k = d_{k+1} = d_{k+2} = | \ul{\bf{d}} |$. Finally, if there are at least two dimension values $a \leq b$ strictly between $1$ and $| \ul{\bf{d}} |$, then replacing $a$ and $b$ by $a-1$ and $b+1$, respectively, changes the sum $\sum_{i=1}^{k+2} d_i(n-d_i)$ by
$$(a-1)(n-a+1) + (b+1)(n-b-1) - a(n-a) - b(n-b) = 2(a-b) - 2,$$
which is a negative number. Thus, we may assume that  $\ul{\bf{d}}$ consists of $r$-many $1$'s, $(s+2)$-many $| \ul{\bf{d}} |$'s (with $s \geq 1$), and at most one number $a$ between $1$ and $| \ul{\bf{d}} |$. Let $\epsilon \in \{ 0, 1 \}$ be the number of times $a$  appears in $\ul{\bf{d}}$.

Now, assuming $\sum_{i=1}^k d_i = r + \epsilon a + s| \ul{\bf{d}} |= n$, we have $\sum_{i=1}^{k+2} d_i = n + 2 | \ul{\bf{d}} |$, and thus
$$\sum_{i=1}^{k+2} d_i (n-d_i) = n^2 + 2n|\ul{\bf{d}}| - \sum_{i=1}^{k+2}d_i^2 = n^2 + (s-2) |\ul{\bf{d}}|^2 + \epsilon a (2|\ul{\bf{d}}|-a) + r(2|\ul{\bf{d}}|-1).$$
Since $| \ul{\bf{d}} |>a > 1$, if $s\geq 2$, then the right hand side is evidently at least $n^2$. On the other hand, if $s=1$, $r+ \epsilon a = n - | \ul{\bf{d}} | \geq | \ul{\bf{d}} |$ and $$-| \ul{\bf{d}} |^2 + (\epsilon a + r) | \ul{\bf{d}} |+ \epsilon a(| \ul{\bf{d}} |-a) + r(| \ul{\bf{d}} |-1) \geq 0.$$  We thus conclude that $\sum_{i=1}^{k+2} d_i(n - d_i) \geq n^2$ as desired.
\end{proof}


\section{Reduction techniques} \label{s:rt}

In this section, we prove a series of lemmas that reduce the density/sparsity problem for certain dimension vectors to the same problem for smaller ones.

\begin{lemma} \label{l:2}
The dimension vector $\ul{\bf{d}} = (d_1, \dots, d_k;n)$ is dense if $\sum_{i=1}^k d_i \leq n$.
\end{lemma}

\begin{proof}
By Lemma \ref{2.7} and adding a linear space of dimension $n - \sum_{i=1}^k d_i$, we may assume that $\sum_{i=1}^k d_i = n$.  Let $X$ be the ambient vector space with basis $\mc{B} = \{ e_1, \dots, e_n \}$. For any $1 \leq i \leq k$, let $\alpha_i := \sum_{j=1}^{i-1} d_i$ and consider the subspace $V_i = \langle e_{\alpha_i +1}, \dots, e_{\alpha_{i+1}} \rangle$ of $X$. Then, in the basis $\mc{B}$, the stabilizer of the configuration $(V_1, \dots, V_k; X)$  consists of block diagonal matrices with blocks of sizes $d_1, \dots, d_k$. Therefore, we have
$\mdim(\stab(V_1, \dots, V_k; X)) = \sum_{i=1}^k d_i^2 -1 = n^2 - 1 - \sum_{i=1}^k d_i(n-d_i).$ By Lemma \ref{l:ml}, we conclude that \ul{\bf{d}} is dense.
\end{proof}

\begin{lemma} \label{l:3}
Let $\ul{\bf{d}} = (a_1, \dots, a_r, b_1, \dots, b_s;n)$ be a dimension vector such that $\sum_{i=1}^r a_i = n-k < n$ and $\sum_{j=1}^s (n-b_j) \leq n-k$. Then \ul{\bf{d}} is dense if $\ul{\bf{d}}' = (a_1, \dots, a_r, b_1 - k, \dots, b_s - k; n - k)$ is dense.
\end{lemma}

\begin{proof}
Let $(V_1, \dots, V_r, U_1, \dots, U_s; X)$ be a generic configuration of vector spaces corresponding to the dimension vector \ul{\bf{d}}. Set $W := V_1 + \dots + V_r$ and $T_j := W \cap U_j$, $1 \leq j \leq s$. By our numerical assumptions, we have $\mdim(W) = n-k$ and $\mdim(T_j) = b_j - k$, for all $j$. Observe that this construction yields a generic configuration with dimension vector $\ul{\bf{d}}'$. Every element $M$ of $\pgl{n}$ that stabilizes the configuration $(V_1, \dots, V_r, U_1, \dots, U_s; X)$, preserves the subspace $W$ and the restriction of $M$ to $W$ stabilizes the configuration $(V_1, \dots, V_r, T_1, \dots, T_s; W)$. Hence, we get a group homomorphism
$$f : \stab(V_1, \dots, V_r, U_1, \dots, U_s; X) \ra \stab(V_1, \dots, V_r, T_1, \dots, T_s; W).$$
Set $Y = \cap_{j=1}^s U_j$ and note that $\mdim(Y) = n - \sum_{j=1}^s (n-b_j) \geq k$. We may choose a basis $\mc{B} = \{ e_1, \dots, e_n \}$ for $X$ such that $W = \langle e_1, \dots, e_{n-k} \rangle,~~ Y = \langle e_{n - \mdim(Y) + 1}, \dots, e_n \rangle.$
Then, with respect to the basis \mc{B}, the kernel of $f$ consists of matrices of the form
$$\left( \begin{array}{ccc}
I_{n - \mdim(Y)} & 0 & 0 \\
0 & I_{\mdim(Y) - k} & A \\
0 & 0 & B
\end{array} \right)$$
where $A$ is a $(\mdim(Y) - k) \times k$ matrix and $B$ is a $k \times k$ matrix. Hence, $\mdim(\mker(f)) = \mdim(Y) \times k$. The dimension vector $\ul{\bf{d}}'$ is dense by our assumption. Therefore,
$$\mdim(\stab(V_1, \dots, V_r, T_1, \dots, T_s; W)) = (n-k)^2 - 1 - \sum_{i=1}^r a_i(n-k-a_i) - \sum_{j=1}^s (b_j - k)(n-b_j).$$
The lemma  follows from Lemma \ref{l:ml} since
$$\mdim(\stab(U_1, \dots, U_r, V_1, \dots, V_s; X)) \leq \mdim(\stab(V_1, \dots, V_r, T_1, \dots, T_s; W)) + \mdim(\mker(f))$$
$$= n^2 - 1 - \sum_{i=1}^r a_i(n-a_i) - \sum_{j=1}^s b_j(n-b_j).$$
\end{proof}

\begin{lemma} \label{l:4}
Let $\ul{\bf{d}}=(a_{11}, \dots, a_{1s_1}, a_{21}, \dots, a_{2s_2}, \dots, a_{r1}, \dots, a_{rs_r}; n)$ be a dimension vector such that $\sum_{j=1}^{s_i} a_{ij} \leq n$ for all $1 \leq i \leq r$.   Then $\ul{\bf{d}}$ is sparse if the dimension vector  $(\sum_{j=1}^{s_1} a_{1j}, \dots, \sum_{j=1}^{s_r} a_{rj}; n)$ is sparse.
\end{lemma}

\begin{proof}
By induction on $r$ and on $s_r$, it suffices to prove that the dimension vector $\ul{\bf{d}} = (a_1, \dots, a_t, b,c; n)$ is sparse if $\ul{\bf{d}}' = (a_1, \dots, a_t, b+c; n)$ is sparse. By Lemma \ref{l:ml}, this would follow from the inequality
\begin{equation} \label{e:ine}
\mdim(\stab(\ul{\bf{d}}')) \leq \mdim(\stab(\ul{\bf{d}})) + 2bc.
\end{equation}
Let $(V_1, \dots, V_t, U, W; X)$ be a generic configuration of vector spaces corresponding to $\ul{\bf{d}}$. Then, the configuration $(V_1, \dots, V_t, U + W; X)$ corresponds to $\ul{\bf{d}}'$. Every element $M$ of $\pgl{n}$ that stabilizes $(V_1, \dots, V_t, U, W; X)$, also stabilizes $(V_1, \dots, V_t, U+W; X)$. Therefore, $\stab(V_1, \dots, V_t, U, W; X)$ is a subgroup of $\stab(V_1, \dots, V_t, U+W; X)$. The map
$$f : \stab(V_1, \dots, V_t, U+W; X) / \stab(V_1, \dots, V_t, U, W; X) \mono Gr(b,U+W) \times Gr(c,U+W),$$
which sends an element $M$ to $(M_{|U+W}U, M_{|U+W}W)$ is injective. This proves the desired inequality \eqref{e:ine} since $\mdim(Gr(b,U+W)) = \mdim(Gr(c, U+W)) = bc$.
\end{proof}

This result allows us to improve slightly on Lemma \ref{l:2}:
\begin{lemma} \label{l:4.5}
The dimension vector $\ul{\bf{d}} = (d_1, \dots, d_k;n)$ is dense if $\sum_{i=1}^k d_i \leq n+1$.
\end{lemma}

\begin{proof} Indeed, we know that the dimension vector $(1^r;n)$ is dense for $r\leq n+1$ (see \S \ref{sec-intro} and \cite[Section 1.6]{har92}), so the dimension vector $\ul{\bf{d}} = (d_1, \dots, d_k;n)$ is dense when $\sum_{i=1}^k d_i \leq n+1$ by Lemma \ref{l:4}.
\end{proof}

\begin{lemma} \label{l:5}
Suppose that an element $M$ of $\pgl{n}$ stabilizes a generic configuration $$(V_1, \dots, V_r; X)$$  corresponding to a dimension vector $\ul{\bf{d}} = (a_1, \dots, a_r; n)$ with $\sum_{i=1}^r a_i = n-k \leq n$. Suppose moreover that $M$ preserves a generic subspace $U$ of $X$ with $\mdim(U) \geq |\ul{\bf{d}}| + k$ and acts on it as the identity. Then $M$ is the identity element.
\end{lemma}

\begin{proof}
For every $1 \leq i \leq r$, set $\alpha_i := \sum_{j=1}^i a_i$, and choose a basis $\mc{B} = \{ e_1, \dots, e_n \}$ for $X$ such that
$$ V_i = \langle e_{\alpha_{i-1}+1}, \dots, e_{\alpha_i} \rangle, \ \ \forall 1 \leq i \leq r,$$
and that $U = \langle f_1, \dots, f_{\mdim(U)-k}, e_{n-k+1}, \dots, e_n \rangle$, with $f_j \in \langle e_1, \dots, e_{n-k} \rangle$. We can express the vectors $f_j$, $1 \leq j \leq \mdim(U) - k$, in terms of the vectors $e_i$, $1 \leq i \leq n-k$, in the  form
$$(f_1, \dots, f_{\mdim(U)-k}) = (E_1, \dots, E_r) \left( \begin{array}{ccc}
A_1 \\
\vdots \\
A_r
\end{array} \right),$$
where $E_i = (e_{\alpha_{i-1}+1}, \dots, e_{\alpha_i})$ and $A_i$ is an $a_i \times (\mdim(U) - k)$ matrix. Since $U$ is assumed to be a generic subspace, each $A_i$ has an invertible full minor. Thus, after suitable change of basis for each  $V_i$, we may assume that for all $1 \leq i \leq r$, the matrix $A_i$ has the form $A_i = (I_{a_i} ~ *)$. Then, for an element $M \in \stab(V_1, \dots, V_r; X)$, $M(f_j) = f_j$ for all $1 \leq j \leq \mdim(U) - k$, implies that $M(e_i) = e_i$ for all $1 \leq i \leq n-k$. Therefore, if $M$ stabilizes the configuration $(V_1, \dots, V_r; X)$ and acts as identity on $U$, $M$ has to be the identity element.
\end{proof}

\begin{lemma} \label{l:6}
Consider the dimension vector $\ul{\bf{d}} = (a_1, \dots, a_r, b, b; n)$ with $b + \sum_{i=1}^r a_i = n$ and $b \leq \frac{n}{2}$. Then \ul{\bf{d}} is dense if the dimension vector $\ul{\bf{d}}' = (a_1, \dots, a_r, b; n-b)$ is dense.
\end{lemma}

\begin{proof}
Let $(V_1, \dots, V_r, U_1, U_2; X)$ be a generic configuration of vector spaces corresponding to $\ul{\bf{d}}$. Set $W := V_1 + \dots + V_r$ and $U' := W \cap (U_1 + U_2)$. Then $\mdim(W) = n-b$ and $\mdim(U') = b$. Observe that this construction yields a generic configuration with dimension vector $\ul{\bf{d}}'$. Consider the map induced by restriction
$$f : \stab(V_1, \dots, V_r, U_1, U_2; X) \ra \stab(V_1, \dots, V_r, U'; W).$$
If the dimension vector $\ul{\bf{d}}'$ is dense, then by Lemma \ref{l:ml},
$$\mdim(\stab(V_1, \dots, V_r, U'; W)) = (n-b)^2 - 1 - \sum_{i=1}^r a_i(n-b-a_i) - b(n-2b)$$
$$= n^2 - 1 - \sum_{i=1}^r a_i(n-a_i) - 2b(n-b).$$
By another application of Lemma \ref{l:ml}, it suffices to prove that the map $f$ is injective. An element $M \in \mker(f)$ preserves $U_1$ and $U_2$ and acts as the identity on $W$, and thus has to be the identity element by Lemma \ref{l:5}.
\end{proof}

\begin{lemma} \label{l:7}
Let $\ul{\bf{d}} = (a_1, \dots, a_r, b; n)$ be a dimension vector with $\sum_{i=1}^r a_i = n$, $|\ul{\bf{d}}| = b$, and $a_i + b \leq n$ for all $1 \leq i \leq r$. Then $\ul{\bf{d}}$ is dense if the dimension vector $\ul{\bf{d}}' = (a_1, \dots, a_r; b)$ is dense.
\end{lemma}

\begin{proof}
Let $(V_1, \dots, V_r, U; X)$ be a generic configuration of vector spaces corresponding to $\ul{\bf{d}}$. For any $1 \leq i \leq r$, set $W_i := V_1 + \dots + \wh{V}_i + \dots + V_r$ and $U_i := W_i \cap U$. Then $\mdim(W_i) = n-a_i$, $\mdim(U_i) = b-a_i$, and we have a homomorphism induced by restriction
$$f : \stab(V_1, \dots, V_r, U; X) \ra \stab(U_1, \dots, U_r; U).$$
By Lemma \ref{l:5}, $f$ has trivial kernel. On the other hand, if $\ul{\bf{d}}'$ is dense, so is its complement $\ul{\bf{d}}'^c = (b-a_1, \dots, b-a_r; b)$, and thus we have:
$$\mdim(\stab(U_1, \dots, U_r; U)) = b^2 -1 - \sum_{i=1}^r a_i(b-a_i)$$
$$= n^2 - 1 - \sum_{i=1}^r a_i(n-a_i) - b(n-b).$$
Therefore, by Lemma \ref{l:ml}, $\ul{\bf{d}}$ is dense as well.
\end{proof}

\begin{lemma} \label{l:8}
Let $\ul{\bf{d}} = (a_1, \dots, a_r, b_1, b_2; n)$ be a dimension vector with $\sum_{i=1}^r a_i = n-k < n$, $k \leq b_1, b_2$, and $b_1 + b_2 = n$. Then $\ul{\bf{d}}$ is dense if the dimension vector $\ul{\bf{d}}' = (a_1, \dots, a_r, b_1 - k, b_2 - k; n-k)$ is dense.
\end{lemma}

\begin{proof}
Let $(V_1, \dots, V_r, U_1, U_2; X)$ be a generic configuration of vector spaces corresponding to \ul{\bf{d}} and set $W := V_1 + \dots + V_r$ and $U_i' := W \cap U_i$ for $i = 1,2$. Then $\mdim(W) = n-k$, $\mdim(U_i') = b_i - k$ for $i = 1,2$. Observe that this construction yields a generic configuration with dimension vector $\ul{\bf{d}}'$. We have the homomorphism induced by restriction
$$f : \stab(V_1, \dots, V_r, U_1, U_2; X) \ra \stab(V_1, \dots, V_r, U_1', U_2'; W).$$
Since any element $M \in \mker(f)$ preserves $U_1$ and $U_2$ and acts as identity on $W$, by  Lemma \ref{l:5}, $f$ is injective. Therefore, if $\ul{\bf{d}}'$ is dense, we have
$$\mdim(\stab(V_1, \dots, V_r, U_1, U_2; X)) \leq \mdim(\stab(V_1, \dots, V_r, U_1', U_2'; W))$$
$$= (n-k)^2 - 1 - \sum_{i=1}^r a_i(n-k-a_i) - \sum_{j=1}^2 (b_j-k)(n-b_j)$$
$$= n^2 - 1 - \sum_{i=1}^r a_i(n-a_i) - \sum_{j=1}^2 b_j(n-b_j).$$
This, together with Lemma \ref{l:ml}, implies that the dimension vector $\ul{\bf{d}}$ is dense.
\end{proof}

\begin{lemma} \label{l:9}
Let $\ul{\bf{d}} = (a_1, \dots, a_r, b_1, b_2; n)$ be a dimension vector with $\sum_{i=1}^r a_i = n-k < n$, $k \leq b_1, b_2$, and $b_1 + b_2 < n$. Set $m = b_1 + b_2 - k$. Then  $\ul{\bf{d}}$ is dense if the dimension vector $\ul{\bf{d}}' = (a_1, \dots, a_r, b_1, b_2; m)$ is dense.
\end{lemma}

\begin{proof}
We prove the assertion in two steps. Let $(V_1, \dots, V_r, U_1, U_2; X)$ be a generic configuration of vector spaces corresponding to \ul{\bf{d}} and set $W := V_1 + \dots + V_r$ and $U := U_1 + U_2$. Also, let $T = U \cap W$ and $U_i' = U_i \cap W$ for $i = 1,2$. Note that $\mdim(W) = n-k$ and $\mdim(T) = b_1 + b_2 - k = m$. Any element in the stabilizer of $(V_1, \dots, V_r, U_1, U_2; X)$ stabilizes $W$, $T$, $U_1'$, and $U_2'$, and thus we get a homomorphism:
$$f : \stab(V_1, \dots, V_r, U_1, U_2; X) \ra \stab(V_1, \dots, V_r, U_1', U_2', T; W).$$
The new twist in this argument is that we are considering a configuration of vector spaces which are not in general position, as $T$ contains $U_1'$ and $U_2'$.

Lemma \ref{l:5} implies that $f$ has trivial kernel. Now assume that $\stab(V_1, \dots, V_r, U_1', U_2', T; W)$ has the expected dimension. Since $T$ contains $U_1'$ and $U_2'$, this expected dimension is
\begin{equation}\label{eqn-l9}
(n-k)^2 - 1 - \sum_{i=1}^r a_i(n-k-a_i) - m(n-k-m) - \sum_{j=1}^2 (b_j - k)(m+k-b_j).
\end{equation}
Using the equalities $m = b_1 + b_2 - k$ and $\sum_{i=1}^r a_i = n-k$, the expression (\ref{eqn-l9}) equals
$$n^2 - 1 - \sum_{i=1}^r a_i(n-a_i) - \sum_{j=1}^2 b_j(n-b_j).$$
Then Lemma \ref{l:ml} would imply that $\ul{\bf{d}}$ is dense.

Hence, it suffices to prove that $\stab(V_1, \dots, V_r, U_1', U_2', T; W)$ has the expected dimension. For every $1 \leq i \leq r$, let $W_i := V_1 + \dots + \wh{V}_i + \dots + V_r$ and $V_i' := T \cap W_i$. Then $(V_1', \dots, V_r', U_1', U_2'; T)$ is a generic configuration of vector spaces corresponding to the complement of the dimension vector $\ul{\bf{d}}'$. Restricting from $W$ to $T$, induces a homomorphism
$$g : \stab(V_1, \dots, V_r, U_1', U_2', T; W) \ra \stab(V_1', \dots, V_r', U_1', U_2'; T),$$
whose kernel is trivial by Lemma \ref{l:5}. If $\ul{\bf{d}}'$ is dense, then so is its complement, and hence the dimension of the image of $g$ is at most
$$m^2 - 1 - \sum_{i=1}^r a_i(m-a_i) - \sum_{j=1}^2 (b_j - k)(m + k - b_j),$$
which using $m = b_1 + b_2 - k$ and $\sum_{i=1}^r a_i = n-k$ simplifies to
$$n^2 - 1 - \sum_{i=1}^r a_i(n-a_i) - \sum_{j=1}^2 b_j(n-b_j)$$ as desired.
\end{proof}

\begin{lemma} \label{l:10}
Consider the dimension vector $\ul{\bf{d}} = (a_1, \dots, a_r; n)$. Assume that there are $k$ elements $i_1, \dots, i_k \in \{ 1, \dots, r \}$ such that $\sum_{j=1}^k a_{i_j} = (k-1)n$ and let $\ul{\bf{d}}'$ be the dimension vector obtained from \ul{\bf{d}} after replacing $a_{i_j}$ with $b_{i_j} := \sum_{\pa{t=1}{t \neq j}}^k a_{i_t} - (k-2)n$. Then $\ul{\bf{d}}$ is dense if and only if $\ul{\bf{d}}'$ is dense.
\end{lemma}

\begin{proof}
Let $(V_1, \dots, V_r; X)$ be a generic configuration of vector spaces corresponding to \ul{\bf{d}}. For every $1 \leq t \leq r$, set $U_t := V_t$ if $t \neq i_1, \dots, i_k$, and $U_t := \cap_{\pa{s=1}{s \neq j}}^k V_{i_s}$ if $t = i_j$ for some $1 \leq j \leq k$. Then, by our numerical assumptions, the configuration $(U_1, \dots, U_r; X)$ corresponds to the dimension vector $\ul{\bf{d}}'$ and  the original vector spaces $V_i$ can be recovered from the vector spaces $U_j$. Indeed, for any $1 \leq j \leq k$, we have
$$\sum_{\pa{t=1}{t \neq j}}^k b_{i_t} = a_{i_j} + (k-2) \sum_{t=1}^k a_{i_t} - (k-1)(k-2)n = a_{i_j}$$
and thus
$$V_{i_j} = \cup_{\pa{t=1}{t \neq j}}^k U_{i_t}.$$
This implies that the density/sparsity of $\ul{\bf{d}}$ is equivalent to that of $\ul{\bf{d}}'$.
\end{proof}

\begin{rem}
The converses of Lemmas \ref{l:3}, \ref{l:6}, \ref{l:7}, \ref{l:8} and \ref{l:9} are easily seen to hold. In each case, the generic subspaces with invariants $\ul{\bf{d}}'$ can be obtained via the construction in the proof. Hence, if $\ul{\bf{d}}'$ is not dense, $\ul{\bf{d}}$ is certainly not dense.
\end{rem}


\section{Dimension vectors with small length}\label{sec-length}

Our goal in this section is to characterize dense dimension vectors of small length using the reduction lemmas. Recall that the length of a dimension vector $\ul{\bf{d}} = (d_1, \dots, d_k; n)$ is defined to be the number $k$. Every dimension vector of length one is dense since the Grassmannian  $Gr(k, n)$ is a quotient of \pgl{n}. In the following result, we show that most dimension vectors with length at most four are dense. Invariants of vectors of length 4 have been studied in \cite{sw}. Here we give a short self-contained treatment which we will use in later sections.

\begin{thm} \label{t:sl}
Let \ul{\bf{d}} be a dimension vector with length $k \leq 4$. Then $\ul{\bf{d}}$ is sparse if and only if $k = 4$ and $\ul{\bf{d}} = (a, b, c, d; n)$ with $a+b+c+d = 2n$.
\end{thm}

\begin{proof}
First, we show that all dimension vectors of length two and three are dense. Let $\ul{\bf{d}} = (a, b; n)$ be a dimension vector of length two. By taking the complement if necessary, we can assume that $a + b \leq n$. Then $\ul{\bf{d}}$ is dense by Lemma \ref{l:2}.

Now let $\ul{\bf{d}} = (a, b, c; n)$ be a dimension vector of length $3$. 
\begin{itemize}
\item{If $n = 2k$ and $a = b = c = k$,  let  $e_1, \dots, e_{2k}$ be a basis of $X$ and let $V_1= <e_i>_{i=1}^k$, $V_2=<e_i>_{i=k+1}^{2k}$ and $V_3= <e_i+e_{i+k}>_{i=1}^k$. Then $\stab(V_1, V_2, V_3; X)$ consists of matrices of the form
    $$M = \left( \begin{array}{cc}
    A & 0 \\
    0 & A
    \end{array} \right).$$
    Since the stabilizer has dimension $k^2-1$, Lemma \ref{l:ml} implies that $\ul{\bf{d}}$ is dense.}
\item{If $a$, $b$, and $c$ are not all equal to $n/2$, by taking the complement and rearranging if necessary, we can assume that $a+b = n-k < n$. If $c \leq k$, \ul{\bf{d}} is dense by Lemma \ref{l:2}. If $c > k$, then, by Lemma \ref{l:3}, the density of $\ul{\bf{d}}$  follows from the density of $(a, b, c-k; n-k)$. We are done since either we are in the first case and $\ul{\bf{d}}$ is dense, or we can inductively continue reducing the ambient dimension $n$.}
\end{itemize}

Finally, let  $\ul{\bf{d}} = (a, b, c, d; n)$ be a dimension vector of length four. We will show that $\ul{\bf{d}}$  is dense if and only if $a + b + c + d \neq 2n$. We begin by studying the case $a + b + c + d = 2n$.
\begin{itemize}
\item{First, the dimension vector $(k, k, k, k; 2k)$ is trivially sparse.}
\item{Next assume  that $a + b + c + d = 2n$, but  $a$, $b$, $c$, and $d$ are not all equal. Then, by taking the complement and rearranging if necessary, we can assume that $a + b < n$. By Lemma \ref{l:3}, the density of $\ul{\bf{d}}$ reduces to the density of $$\ul{\bf{d}}' = (a, b, a+b+c-n, a+b+d-n; a+b).$$ Since
    $a+b+(a+b+c-n)+(a+b+d-n) = 2(a+b),$
by induction on the dimension of the ambient space, $\ul{\bf{d}}'$ is sparse. Therefore,  $\ul{\bf{d}}$ is sparse.}
\end{itemize}

Now we suppose that $a + b + c + d \neq 2n$ and show by induction on $n$ that $\ul{\bf{d}}$ is dense. By taking the complement if necessary, we can assume that $a + b + c + d < 2n$.
\begin{itemize}
\item{If $a + b + c + d \leq n$, then \ul{\bf{d}} is dense by Lemma \ref{l:2}.}
\item{If sum of the two larger dimensions is bigger than $n$, we take the complement and rearrange so that $a + b + c + d > 2n$ and $a + b = n-k < n$. This implies that
    $$(n-c) + (n-d) = 2n - (c + d) < a + b = n-k.$$
    Thus, by Lemma \ref{l:3}, \ul{\bf{d}} is dense if $\ul{\bf{d}}' = (a, b, c-k, d-k; n-k)$ is dense. Since
    $$a + b + (c-k) + (d-k) = (a + b + c + d) - 2k > 2(n-k),$$
    by induction on $n$, $\ul{\bf{d}}'$ is dense and we are done. In the remaining cases, we may assume that $n < a + b + c + d < 2n$, $a + b = n-k < n$, and $c + d \leq n$ (suppose for simplicity that we have ordered the dimensions so that $a \leq b \leq c \leq d$).}
\item{If $c < k$, then $a + b + c = n - t < n$. Lemma \ref{l:3} implies that $\ul{\bf{d}}$ is dense if $\ul{\bf{d}}' = (a, b, c, d-t; n-t)$ is dense. Since
    $$a + b + c + (d-t) = (n-t) + (d - t) < 2(n-t),$$
  $\ul{\bf{d}}'$ is dense by induction. Therefore, in the following cases we may assume that $k \leq c,d$.}
\item{If $c + d = n$, then by Lemma \ref{l:8}, $\ul{\bf{d}}$ is dense if $\ul{\bf{d}}' = (a, b, c-k, d-k; n-k)$ is dense. Since
    $$a + b + (c-k) + (d - k) = (a + b + c + d) - 2k < 2(n-k),$$
    $\ul{\bf{d}}'$ is dense by induction.}
\item{Finally, assume that $c + d < n$. Then Lemma \ref{l:9} implies that $\ul{\bf{d}}$ is dense if $(a,b,c,d;m)$ is dense, where $m=c+d-k$. By taking the complement, the density of the latter is equivalent to the density of $\ul{\bf{d}}' = (m-a, m-b, c-k, d-k; m)$. Since
    $$(m-a) + (m-b) + (c-k) + (d-k) = 2m - (a+b) + (c+d) - 2k$$ $$ < 2m - (n-k) + n - 2k = 2m - k < 2m,$$
    by induction, we conclude that $\ul{\bf{d}}'$ is dense.}
\end{itemize}
\end{proof}

\begin{cor} \label{c:sl}
Let $\ul{\bf{d}} = (d_1, \dots, d_k; n)$ be a dimension vector such that there is a subsequence of $d_1, \dots, d_k$ with total sum $2n$. Then $\ul{\bf{d}}$ is sparse.
\end{cor}

\begin{proof}
This follows from Theorem \ref{t:sl} and Lemma \ref{l:4}.
\end{proof}


\section{Dimension vectors with small size}\label{sec-size}

 In this section we study dimension vectors of a given size, without any restriction on the length. We illustrate how the techniques we developed in Section \ref{s:rt} can be employed in classifying dense dimension vectors, and fully classify dense dimension vectors with size at most four. We begin by the following useful consequence of Lemma \ref{l:nl}.

\begin{prop} \label{p:r}
Let $\ul{\bf{d}} = (d_1, \dots, d_r; n)$ be a dimension vector of size $| \ul{\bf{d}} | = k$ with $2k \leq n$. Then \ul{\bf{d}} is either trivially sparse, or is dense by Lemma \ref{l:2}, or the density/sparsity problem for $\ul{\bf{d}}$ can be reduced to the density/sparsity problem for a dimension vector with smaller ambient dimension.
\end{prop}

\begin{proof}
Let $\ul{\bf{d}}$ be as in the statement of the proposition and suppose that it is not trivially sparse. Then, assuming that the $d_i$'s are arranged in increasing order, by Lemma \ref{l:nl}, we have
$$\sum_{i=1}^{r-2} d_i < n.$$
This leads to the following cases:
\begin{enumerate}
\item[(I)]{If $\sum_{i=1}^r d_i \leq n$, then \ul{\bf{d}} is dense by Lemma \ref{l:2}.}
\item[(II)]{If $\sum_{i=1}^{r-1} d_i = n-k < n$ and $k < d_r$, then we can apply Lemma \ref{l:3} to reduce the density problem for \ul{\bf{d}} to the density problem for $(d_1, \dots, d_{r-1}, d_r - k; n-k)$.}
\item[(III)]{If $\sum_{i=1}^{r-2} d_i = n-k < n$ and $k \leq d_{r-1}, d_r$, then by  Lemma \ref{l:8} or Lemma \ref{l:9} (depending on whether $d_{r-1} + d_r$ is equal to or less than $n$), we reduce the density problem of $\ul{\bf{d}}$ to one with smaller ambient dimension.}
\end{enumerate}
\end{proof}

Suppose  we want to determine if a dimension vector $\ul{\bf{d}} = (d_1, \dots, d_r; n)$ with $| \ul{\bf{d}} | \leq t$ is dense or sparse. By Proposition \ref{p:r}, if $n \geq 2 | \ul{\bf{d}} |$ and the problem is not trivially answered, this problem can be reduced to one with smaller ambient dimension. Lemmas \ref{l:2}, \ref{l:3}, \ref{l:8} and \ref{l:9} used in the proof of Proposition \ref{p:r} preserve the size of the dimension vectors. Hence, we may assume that $| \ul{\bf{d}}| < n < 2 | \ul{\bf{d}} |$. On the other hand, as the minimum value for $d_i(n-d_i)$, with $1 \leq d_i \leq n-1$, is $n-1$, if $r > n+1$ we have:
$$\sum_{i=1}^r d_i (n-d_i) > (n-1)(n+1) = n^2 - 1,$$
and hence $\ul{\bf{d}}$ is trivially sparse. Therefore, we may assume that $r \leq n+1$. Furthermore, if $r = n+1$, the same calculation shows that if any of $d_i$'s are not equal to $1$ or $n-1$, then $\ul{\bf{d}}$ is trivially sparse. Hence, the density/sparsity problem for dimension vectors of bounded size is reduced to a finite collection of low-dimensional cases. Here we will demonstrate how to use the results of the previous sections to classify all dense vectors with $| \ul{\bf{d}} | \leq 4$.

In the following, we will frequently use the exponential notation for dimension vectors (for example the dimension vector $(1,1,2,2,2,4;7)$ will be denoted as $(1^2, 2^3, 4; 7)$). We will also use the following simple observation.

\begin{lemma} \label{l:m}
For any $n \geq 2$, the dimension vector $(1^n, n-1; n)$ is dense.
\end{lemma}

\begin{proof}
Let $X$ be an $n$-dimensional vector space with basis $\mc{B} = \langle e_1, \dots, e_n \rangle$. Consider the configuration $(V_1, \dots, V_n, U; X)$ of vector spaces with $V_i = \langle e_i \rangle$ for $1 \leq i \leq n$, and $U = \{ (x_1, \dots, x_n) : \sum_{i=1}^n x_i = 0 \}$ (the coordinates are with respect to the basis \mc{B}). Then $(V_1, \dots, V_n, U; X)$ corresponds to the dimension vector $(1^n, n-1; n)$ and it is easy to check that $\stab(V_1, \dots, V_n, U; X)$ is trivial. Hence, we are done by Lemma \ref{l:ml}.
\end{proof}

We now begin the classification of dimension vectors of size  at most four.

\vspace{5mm}
{\noindent \large{$| \ul{\bf{d}} | = 1$:}} This case is trivial. The dimension vector has the form $\ul{\bf{d}} = (1^r; n)$, which is trivially sparse if $r > n+1$ and is dense otherwise (see \S \ref{sec-intro} and \cite[Section 1.6]{har92}).

\vspace{5mm}
{\noindent \large{$| \ul{\bf{d}} | = 2$:}} We first consider the case $n=3$, where Proposition \ref{p:r} does not apply. In this case $r\leq 4$, and by Theorem \ref{t:sl} all vectors of the form $(1^a,2^b;3)$ with $a+b\leq 4 $ are dense except for $(1^2,2^2;3)$.

Now assume that $n\geq 4$. A dimension vector $(1^a,2^b;n)$ is dense if $a+2b\leq n+1$ by Lemma \ref{l:4.5} and trivially sparse if $a+2b\geq n+4$ by Lemma \ref{l:nl}. If $a+2b=n+2$, then we check that $(1^a,2^b;n)$ is trivially sparse unless $a\leq 3$. If this holds, we use either Lemma \ref{l:8} (if $b\geq 2$ and $n=4$) or Lemma \ref{l:9} (if $b=1$ or $b\geq 2$ and $n\geq 5$) to reduce to the dimension vector $(1^a;2)$, which is dense if $a\leq 3$. If $a+2b=n+3$, then $(1^a,2^b;n)$ is trivially sparse unless $a\leq 5-n$, so $n\leq 5$ and we get only two vectors $(1,2^3;4)$ and $(2^4;5)$, both of which are dense by Theorem \ref{t:sl}.

Putting all this together, we obtain the following list of all dense vectors of size 2, listed by excess dimension:

\begin{itemize}

\item $(1^a,2^b;n)$ with $a+2b\leq n+1$.

\item $(1^a,2^b;n)$ with $a+2b=n+2$ and $a\leq 3$.

\item Finitely many vectors with $a+2b\geq n+3$: $(2^3;3)$, $(1,2^3;3)$, $(2^4;3)$, $(1,2^3;4)$ and $(2^4;5)$.
\end{itemize}

\vspace{5mm}
{\noindent \large{$| \ul{\bf{d}} | = 3$:}} We first consider the cases $n=4$ and $n=5$ where Proposition \ref{p:r} does not apply. We have the following cases:
\begin{itemize}
\item{$n=4$, $r \leq 4$: By Theorem \ref{t:sl}, all dimension vectors are dense except for $(1^2,3^2;4)$ and $(1,2^2,3;4)$.}
\item{$n=4$, $r=5$: Since $r = n+1$, $\ul{\bf{d}}$ is trivially sparse unless it is of the form $(1^a, 3^b; 4)$ with $a+b = 5$. On the other hand, if $2 \leq a,b$, then $\ul{\bf{d}}$ dominates $(1^2,3^2;4)$ and hence is sparse. The remaining cases are $(1^4,3;4)$ (dense by Lemma \ref{l:m}), $(1,3^4; 4)$ (dense by taking the complement), and $(3^5; 4)$ (dense by taking the complement).}
\item{$n=5$, $r \leq 4$: Theorem \ref{t:sl} implies that all dimension vectors are dense except for $(1,3^3;5)$ and $(2^2,3^2;5)$.}
\item{$n=5$, $r = 5$: All dimension vectors in this case are trivially sparse except for the following three, which can be reduced to smaller dimension vectors that are considered in the previous cases:
    $$(1^4,3;5) \xra{\text{Lemma \ref{l:3}}} (1^4, 2; 4) ~{\text{(trivially sparse)}},$$
    $$(1^3,2,3;5) \xra{\text{Lemma \ref{l:8}}} (1^4; 3) ~{\text{(dense)}},$$
    $$(1^3, 3^2; 5) \xra{\text{complement}} (2^2, 4^3; 5) \xra{\text{Lemma \ref{l:3}}} (2^2, 3^3; 4) ~{\text{(trivially sparse)}}.$$}
\item{$n=5$, $r=6$: All dimension vectors in this case are trivially sparse.}
\end{itemize}

We now assume that $n\geq 6$. A dimension vector $(1^a,2^b,3^c;n)$ is dense if $a+2b+3c\leq n+1$ by Lemma \ref{l:4.5} and trivially sparse if $a+2b+3c\geq n+6$ by Lemma \ref{l:nl}. Hence we need to consider the following cases:
\begin{itemize}
\item $a+2b+3c=n+2$. Here the dimension vector $(1^a,2^b,3^c;n)$ is not trivially sparse if $a\leq 3c+3$. In this case we apply Lemma \ref{l:3} to reduce to the dimension vector $(1^a,2^{b+1},3^{c-1};n-1)$, and hence by induction to the vector $(1^a,2^{b+c};n-c)$, which is dense if and only if $a\leq 3$.

\item $a+2b+3c=n+3$. In this case the dimension vector $(1^a,2^b,3^c;n)$ is not trivially sparse if $a+b\leq 4$. If this holds, then we use Lemma \ref{l:8} (when $c\geq 2$ and $n=6$) or Lemma \ref{l:9} (when $c=1$ or $n\geq 7$) to reduce to the dimension vector $(1^a,2^b;3)$, which is dense if $a+b\leq 4$ and $(a,b)\neq (2,2)$.

\item $a+2b+3c=n+4$. The dimension vector $(1^a,2^b,3^c;n)$ is not trivially sparse if $3a+4b+3c\leq 15$, so there are only finitely many possibilities. If the length $r=a+b+c\leq 4$, then the only possibilities are $(1,3^3;6)$, $(2^2,3^2;6)$ $(2,3^3;7)$, $(3^4;8)$, which are all dense by Theorem \ref{t:sl}. If $a+b+c=5$ then $b=0$, giving the three vectors $(1^2,3^3;7)$, $(1,3^4;9)$ and $(3^5;11)$.  By Lemma \ref{l:9} these vectors reduce to $(1^2,3^3;4)$, $(1,3^4;4)$ and $(3^5;4)$, respectively, so $(1^2,3^3;7)$ is sparse and $(1,3^4;9)$ and $(3^5;11)$ are dense.

\item $a+2b+3c=n+5$. In this case the dimension vector $(1^a,2^b,3^c;n)$ is trivially sparse unless $a+b\leq 7-n$, so the only possibilities are $(2,3^3;6)$ and $(3^4;7)$, which are dense by Theorem \ref{t:sl}.
\end{itemize}

Putting all this together, we obtain a complete list of dense vectors of size 3, listed by excess dimension.
\begin{itemize}
\item $(1^a,2^b,3^c;n)$ with $a+2b+3c\leq n+1$.
\item $(1^a,2^b,3^c;n)$ with $a+2b+3c=n+2$ and $a\leq 3$.
\item $(1^a,2^b,3^c;n)$ with $a+2b+3c=n+3$, $a+b\leq 4$ and $(a,b)\neq (2,2)$.
\item Finitely many vectors with $a+2b+3c\geq n+4$: $(2,3^2;4)$, $(2^3,3;4)$, $(1,2,3^2;4)$, $(3^3;4)$, $(1,3^3;4)$, $(2,3^3;4)$, $(3^4;4)$, $(1,3^4;4)$, $(2^3,3;5)$, $(1,2,3^2;5)$, $(3^3;5)$, $(1,3^3;5)$, $(2,3^3;5)$, $(3^4;5)$, $(1,3^3;6)$, $(2^2,3^2;6)$, $(2,3^3;6)$, $(2,3^3;7)$, $(3^4;8)$, $(1,3^4;9)$ and $(3^5;11)$.
\end{itemize}

Before proceeding further, we generalize the method that we used above. For a dimension vector of the form $(1^{e_1},\ldots,k^{e_k};n)$ with excess dimension $\sum_{i=1}^k ie_i-n\leq k$, we can reduce to a vector of smaller size.

\begin{thm} \label{th:l+1}Let $ \ul{\bf{d}}=(1^{e_1},\ldots,k^{e_k},n)$ be a dimension vector with total dimension $\sum_{i=1}^k ie_i=n+l+1$, where $l<k$ and $e_k>0$. Then $\ul{\bf{d}}$ is dense if and only if the dimension vector $(1^{e_1},\ldots,l^{e_l};l+1)$ is dense.

\end{thm}
\begin{proof}
First assume that $k\geq l+2$. In this case we can repeatedly apply Lemma \ref{l:3} to replace each such $k$ with $l+1$, without changing the excess dimension. Hence we can assume that $k=l+1$, and we have reduced $\ul{\bf{d}}$ to the dimension vector $(1^{e_1},\ldots,l^{e_l}, (l+1)^{f};m)$, where $f=\sum_{i=l+1}^{k} e_i>0$ and $\sum_{i=1}^l ie_i+(l+1)f=m+l+1$. If $f\geq 2$ and $m=2l+4$, then Lemma \ref{l:8} reduces to the dimension vector $(1^{e_1},\ldots, l^{e_l};l+1)$ (we drop extra components of dimensions $0$ and $l+1$). If $f\geq 2$ and $m\geq 2l+5$, then Lemma \ref{l:9} also reduces to the dimension vector $(1^{e_1},\ldots, l^{e_l};l+1)$. Finally, if $f=1$, then either $\ul{\bf{d}}$ has no more than two components (in which case the theorem holds trivially), or we can use Lemma \ref{l:9} to reduce to the dimension vector $(1^{e_1},\ldots,l^{e_l};l+1)$, which completes the proof.

\end{proof}

It follows that we've reduced the classification problem of dense dimension vectors $\ul{\bf{d}}=(a_1,\ldots,a_r;n)$ of size $|\ul{\bf{d}}|=l$ and total dimension $\sum_{i=1}^{r} a_i\leq n+l+1$ to the classification of all dense vectors in ambient dimension $l+1$. The following lemma shows that there is only a finite number of additional cases to consider.

\begin{lemma} \label{l:fi}
For a given $l$, there are finitely many dense dimension vectors $(1^{e_1},\ldots,l^{e_l};n)$ having excess dimension $\sum_{i=1}^l i e_i-n\geq l+1$.
\end{lemma}
\begin{proof}
We can assume that $n\geq 2l$, since there are finitely many dense dimension vectors in a given ambient dimension. Assume that $\sum_{i=1}^l i e_i= n+k+1$ where $k\geq l$. By Lemma \ref{l:nl}, if $k \geq 2l$, the dimension vector is trivially sparse. Hence, it suffices to prove that there are finitely many dense dimension vectors for each $l \leq k < 2l$. The dimension vector $(1^{e_1},\ldots,l^{e_l};n)$ is trivially sparse unless
$$
\displaystyle\sum_{i=1}^l e_i i(n-i)\leq n^2-1.
$$
Using $\sum_{i=1}^l i e_i= n+k+1$, we can reexpress this inequality as $$n(n+k+1) - \sum_{i=1}^l i^2 e_i = n^2 - (k+1)^2 + \sum_{i=1}^l (k+1-i)ie_i \leq n^2-1.$$ Hence, $$\sum_{i=1}^l  (k+1-i)i e_i \leq k(k+2).$$ Since the coefficient of each $e_i$ is positive, there are finitely many such dense dimension vectors.
\end{proof}

We now give the full classification of dense dimension vectors of size $4$.

{\noindent \large{$| \ul{\bf{d}} | = 4$:}} We first consider all cases for which Proposition \ref{p:r} does not apply, namely when the ambient dimension $n$ is $5$,  $6$, or $7$. Then we consider all possibilities for the length $r$ of \ul{\bf{d}}.
\begin{itemize}
\item{$n=5$, $r \leq 4$: Theorem \ref{t:sl} implies that in this case all dimension vectors are dense except for $(1^2,4^2;5)$, $(1,2,3,4;5)$, and $(2^3,4;5)$.}
\item{$n=5$, $r=5$: We consider those dimension vectors that are not trivially sparse (up to taking complement) one at a time in the following list and show how they can be reduced to smaller vectors.
    $$(1^4,4;5) \xra{\text{Lemma \ref{l:3}}} (1^4, 3;4) ~{\text{(dense)}},$$
    $$(1^3, 2, 4; 5) \xra{\text{complement}} (1,3,4^3;5) \xra{\text{Lemma \ref{l:3}}} (1, 3^4; 4) ~{\text{(dense)}},$$
    $$(1^3, 3, 4; 5), (1^2, 2^2, 4; 5) \xra{\text{Lemma \ref{l:4}}} (1^2,4^2; 5) ~\text{(sparse)},$$
    $$(1^2, a, 4^2; 5), ~\text{for}~ 1 \leq i \leq 4, ~\text{dominate}~ (1^2,4^2;5) ~\text{(sparse)},$$
    $$(1^2,2,3,4;5) \xra{\text{Lemma \ref{l:3}}} (1^2,2^2,3;4) ~{\text{(trivially sparse)}},$$
    $$(1^2,3^2,4;5) \xra{\text{Lemma \ref{l:10}}} (1^3, 2^2; 5) ~\text{(dense)},$$}
\item{$n=5$, $r=6$: All dimension vectors in this case are trivially sparse unless they have the form $(1^a, 4^b; 5)$ for $a+b=6$. Also, if $2 \leq a,b$, then the resulting dimension vector dominates $(1^2,4^2; 5)$ and hence is sparse. Thus the only dense dimension vectors in this case are $(1^5, 4; 5)$ (by Lemma \ref{l:m}), $(1, 4^5; 5)$, and $(4^6; 5)$ (consider the complement).}
\item{$n=6$, $r \leq 4$: By Theorem \ref{t:sl}, all dimension vectors in this case are dense except for $(1,3,4^2;6)$, $(2^2,4^2;6)$, and $(2,3^2,4;6)$.}
\item{$n=6$, $r=5$: Here is a list of dimension vectors that are not trivially sparse (up to taking complement) and how to reduce them to smaller ones which are considered in previous cases.
    $$(1^4, 4; 6) \xra{\text{Lemma \ref{l:3}}} (1^4, 2; 4) ~{\text{(trivially sparse)}},$$
    $$(1^3, 2, 4; 6) \xra{\text{Lemma \ref{l:3}}} (1^3, 2, 3; 5) ~{\text{(dense by the case $(|\ul{\bf{d}}|, n, r) = (3,5,5)$)}},$$
    $$(1^3, 3, 4; 6) \xra{\text{complement}} (2,3, 5^3; 6) \xra{\text{Lemma \ref{l:3}}} (2,3, 4^3; 5) ~{\text{(dense by taking complement)}},$$
    $$(1^3, 4^2; 6) \xra{\text{complement}} (2^2, 5^3; 6) \xra{\text{Lemma \ref{l:3}}} (2^2, 3^3; 4) ~{\text{(trivially sparse)}},$$
    $$(1^2,2^2,4;6) \xra{\text{Lemma \ref{l:8}}} (1^2,2^2;4) ~{\text{(dense by Theorem \ref{t:sl})}},$$
    $$(1^2, 2, 3, 4; 6) \xra{\text{complement}} (2,3,4,5^2; 6) \xra{\text{Lemma \ref{l:3}}} (2,3^2, 4^2; 5) ~{\text{(trivially sparse)}},$$
    $$(1^2, 2, 3, 4; 6), (1^2, 3^2, 4; 6) ~\text{(sparse by Corollary \ref{c:sl})},$$
    $$(1^2, 3, 4^2; 6) \xra{\text{Lemma \ref{l:3}}} (1^2, 3^3; 5) ~{\text{(trivially sparse)}},$$
    $$(1^2, 4^3; 6) \xra{\text{Lemma \ref{l:10}}} (1^2, 2^3; 6) \xra{\text{Lemma \ref{l:6}}} (1^2,2^2;4) ~\text{(dense by Theorem \ref{t:sl})}.$$}
\item{$n=6$, $r=6$: All dimension vectors in this case are trivially sparse except for $(1^5, 4; 6)$ which can be reduced by Lemma \ref{l:3} to the smaller vector $(1^5, 3; 5)$.}
\item{$n=6$, $r=7$: All dimension vectors in this case are trivially sparse.}
\item{$n=7$, $r \leq 4$: We know by Theorem \ref{t:sl} that all dimension vectors in this case are dense except for $(2, 4, 4, 4; 7)$ and $(3, 3, 4, 4; 7)$.}
\item{$n=7$, $r=5$: Let $\ul{\bf{d}} = (d_1, \dots, d_4, 4; 7)$, with $1 \leq d_i \leq 4$, be a vector in this category. It can be easily seen that either $\sum_{i=1}^3 d_i < 7$ or \ul{\bf{d}} is trivially sparse. If $\sum_{i=1}^3 d_i < 7$, an argument analogous to the proof of Proposition \ref{p:r} implies that the dimension vector $\ul{\bf{d}}$ can be reduced to a smaller one.}
\item{$n=7$, $r=6$: All dimension vectors in this category are trivially sparse except for the following ones:
    $$(1^5, 4; 7) \xra{\text{Lemma \ref{l:3}}} (1^5, 2; 5) ~{\text{(trivially sparse)}},$$
    $$(1^4, 2, 4; 7) \xra{\text{Lemma \ref{l:3}}} (1^4, 2, 3; 6) ~{\text{(trivially sparse)}},$$
    $$(1^4, 3, 4; 7) \xra{\text{Lemma \ref{l:8}}} (1^5; 4) ~{\text{(dense)}},$$
    $$(1^4, 4^2; 7) \xra{\text{complement}} (3^2, 6^4; 7) \xra{\text{Lemma \ref{l:3}}} (3^2, 5^4; 6) ~{\text{(trivially sparse)}}.$$}
\item{$n=7$, $r=7$: All dimension vectors in this case are trivially sparse except for $(1^6, 4; 7)$ which can be reduced by Lemma \ref{l:3} to the smaller dimension vector $(1^6, 3; 6)$.}
\item{$n=7$, $r=8$: All dimension vectors in this category are trivially sparse by Lemma \ref{l:m}.}
\end{itemize}

We can now use Theorem \ref{th:l+1} and Lemma \ref{l:fi} to give a complete list of dense dimension vectors of size $4$, listed by excess dimension:
\begin{itemize}
\item $(1^a,2^b,3^c,4^d;n)$ with $a+2b+3c+4d\leq n+1$.
\item $(1^a,2^b,3^c,4^d;n)$ with $a+2b+3c+4d=n+2$ and $a\leq 3$.
\item $(1^a,2^b,3^c,4^d;n)$ with $a+2b+3c+4d=n+3$, $a+b\leq 4$ and $(a,b)\neq (2,2)$.
\item $(1^a,2^b,3^c,4^d;n)$ with $a+2b+3c+4d=n+4$ and such that $(1^a,2^b,3^c;4)$ is dense. This means that either $a+b+c\leq 3$, or $a+b+c=4$ and $a+2b+3c\neq 8$, or $a+c=5$ and $b=0$ with either $a\leq 1$ or $c\leq 1$.
\item A finite set of dimension vectors $(1^a,2^b,3^c,4^d;n)$ with $a+2b+3c+4d\geq n+5$. For $n\leq 7$ all these vectors are given above, and for $n\geq 8$ these vectors can be found by solving the inequality in Lemma \ref{l:fi} and using Proposition \ref{p:r} to reduce to known cases.
\end{itemize}

We summarize the results of this section. To classify all dense vectors of size $l$, it is necessary to first classify all dense vectors of ambient dimension up to $l+1$. These vectors generate infinite families of dense vectors of size $l$ by Theorem \ref{th:l+1}, having excess dimension at most $l$. There are finitely many dense vectors having excess dimension greater than $l$. For ambient dimensions between $l+1$ and $2l-1$ these need to be found by hand, and the remaining ones can be found by Lemma \ref{l:fi} and Proposition \ref{p:r}.


\section{The balanced case}\label{s:ed}

In this section, we study dimension vectors where $|d_i - d_j| < 3$. We first need a definition.
\begin{defi}
We say that a  dimension vector $\ul{\bf{d}} = (a_1, \dots, a_k; n)$ can be reduced to the  dimension vector $\ul{\bf{d}}'=(a_1', \dots, a_k'; n')$ by Lemma \ref{l:9} if there exists a sequence of vectors $$\ul{\bf{d}} = \ul{\bf{d}}_0  \rightarrow \ul{\bf{d}}_1 \rightarrow \cdots \rightarrow \ul{\bf{d}}_m = \ul{\bf{d}}',$$ where each vector is obtained from the previous one by either taking the complement or applying Lemma \ref{l:9}.
\end{defi}

\begin{thm}\label{t:ed}
Let $\ul{\bf{d}} = ((k-2)^a, (k-1)^b, k^c; n)$ be a dimension vector not satisfying $a+b+c=4$ and $2n = (a+b+c)k -2a-b$. Assume  $\ul{\bf{d}} $ cannot be reduced by Lemma \ref{l:9} to one of the dimension vectors  $(1^2, 2^2, 3^c; 3c+3)$, $(1^3,3^2;4)$, $(1^4,3;5)$, or $(1^3,3^2;5)$.
Then $\ul{\bf{d}}$ is dense if and only if it is not trivially sparse.
\end{thm}

\begin{rem}
The exceptions listed in the theorem are sparse by Theorem \ref{t:sl} and  the case $| \ul{\bf{d}}|=3$ in \S \ref{sec-size}. When $a=b=0$, the theorem says that equidimensional vectors are dense if and only if they are not trivially sparse and we recover Popov's  \cite[Theorem 3]{p} in the case of $Gr(k,n)$. When $a=0$, the theorem says that dimension vectors with $|d_i - d_j| \leq 1$ are dense if and only if they are not trivially sparse and are not of the form $((k-1)^2, k^2; 2k-1)$.
\end{rem}

\begin{proof}
First, assume $a=b=0$ and $c \not= 0$. After taking complements if necessary, we may assume that $k \leq \frac{n}{2}$. Write $n = mk + r$ with $0 \leq r < k$. 

$\bullet$ If $c> m+2$ or if $c=m+2$ and $r=0$, then  $\ul{\bf{d}}$ is trivially sparse by Lemma \ref{l:nl}.  

$\bullet$ If $c = m+1$ and $r=0$,  then $m-2$ applications of Lemma \ref{l:6}, reduces the density to the density of $(k^3, 2k)$.  The latter is dense by Theorem \ref{t:sl}. 

$\bullet$ If $c=m+1$ and $r>0$, $m+1$ applications of Lemma \ref{l:3}  reduces the density to that of $((k-r)^{m+1}, m(k-r))$. The latter is dense by the previous case. Hence, if $c \leq m+1$, the vector is dense by Lemma \ref{2.7}.

$\bullet$ Finally, if $c= m+2$ and $r>0$, applying Lemma \ref{l:9} and taking complements reduces the density to the density of $((k-r)^{m+2}; 2k-r)$. An easy calculation shows that  this vector is trivially sparse if and only if the original vector is. Since $2k-r < n$, we reduce the problem to one with strictly smaller ambient dimension. By induction on $n$, we conclude that $(k^c; n)$ is dense if and only if it is not trivially sparse.

Next, suppose $a=0$, but $bc \not= 0$. If $n \not= 2k-1$,  after possibly taking complements, we may assume that $k \leq \frac{n}{2}$. If a dimension vector $\left(\left(\frac{n-1}{2}\right)^b, \left(\frac{n+1}{2}\right)^c; n\right)$ is not trivially sparse, then $$(b+c)\left(\frac{n^2}{4} - \frac{1}{4} \right) \leq n^2 -1, \ \ \  \mbox{hence} \ \ b+c \leq 4.$$  By Theorem \ref{t:sl}, the only such sparse dimension vectors  are $((k-1)^2, k^2; 2k-1)$. 

Now we may assume that $k \leq \frac{n}{2}$. 

$\bullet$ If $b(k-1) + ck \leq n+1$, then $\ul{\bf{d}}$ is dense  by Lemma \ref{l:4.5}. 

$\bullet$  If  $b(k-1) + (c-1)k + r = n$ for some $0 < r < k$, $c$ applications of Lemma \ref{l:3}, reduces the density of $\ul{\bf{d}}$ to the density of $((k-r)^c, (k-1)^b; n-cr)$. If $r=1$, we reduce to the equidimensional case. Otherwise, another $b$ applications of Lemma \ref{l:3} reduces the density of $\ul{\bf{d}}$ to the density of $\ul{\bf{d}}'=((k-r)^{b+c}; n-cr-b(r-1))$. Since $n-cr -b(r-1) = (b+c-1)(k-r)$, the vector is not trivially sparse.  By the equidimensional case, we conclude that $\ul{\bf{d}}$ is dense. 

$\bullet$  If $b(k-1) + (c-1)k =n$, an easy calculation shows that if $\ul{\bf{d}}$ is not trivially sparse, then $b \leq k+1$. By $(c-1)$ applications of Lemma \ref{l:6}, we may assume that $c=1$. By Lemma \ref{l:7}, we reduce the density of $\ul{\bf{d}}$ to the density of $((k-1)^b;k)$. The latter is dense since $b \leq k+1$.

$\bullet$ Finally, by Lemma \ref{l:nl}, we may  assume that either
\begin{enumerate}
\item $c \geq 2$ and $b(k-1) + (c-2)k + r = n$ for some $0 < r < k$; or
\item $c=1$ and $(b-1)(k-1) + r = n$ for some $0 < r < k-1$.
\end{enumerate}
By Lemma \ref{l:9} and taking the complement, the density of $\ul{\bf{d}}$ is equivalent to the density of  $((k-r)^c, (k-r+1)^b; 2k-r)$ in  Case (1) and $(k-r-1,(k-r)^b; 2k-r-1)$ in Case (2). Since $k \leq \frac{n}{2}$, both cases have strictly smaller ambient dimension. Moreover, these vectors are trivially sparse if and only if $\ul{\bf{d}}$ is trivially sparse. Furthermore, if $2(k-r) + 1 = 2k-r$ and $b=c=2$, we deduce that $r=1$ and $n= 2k-1$. Hence, the reduction produces a dimension vector where our induction hypotheses apply. By induction, we conclude that a dimension vector of the form $((k-1)^b, k^c; n)$ which is not trivially sparse and of the form $((k-1)^2, k^2; 2k-1)$ is dense.

Now we are ready to analyze the case $ac \not= 0$. Unless $\frac{n}{2}<k \leq \lceil \frac{n}{2} \rceil +1$,  after possibly taking the complement, we may assume $k \leq \frac{n}{2}$. We first need to analyze the cases where $\frac{n}{2}<k \leq \lceil \frac{n}{2} \rceil +1$. If $n$ is even and $k = \frac{n}{2} +1$,  we have that $$(a+c)\left(\frac{n^2}{4}-1\right) + b \frac{n^2}{4} \leq n^2 -1.$$ Hence, either $a+b+c\leq4$ or $n\leq 4$. By Theorem \ref{t:sl} and \S \ref{sec-size}, the only sparse dimension vectors which are not trivially sparse have the form $((k-1)^2, (k+1)^2; 2k)$, $(k-1, k^2, k+1; 2k)$,  $(1^3, 3^2; 4)$ or $(1^2, 3^3;4)$.  If $n$ is odd and $k = \frac{n+1}{2}$ or $\frac{n+3}{2}$, by passing to the complement if necessary, we may take $k = \frac{n+1}{2}$. Using the dimension estimate $$a\left(\frac{n^2}{4}-\frac{9}{4}\right) + (b+c) \left(\frac{n^2}{4}- \frac{1}{4}\right) \leq n^2 -1,$$ we conclude that either $a+b+c \leq 4$ or $n\leq 5$. By Theorem \ref{t:sl} and \S \ref{sec-size}, the only sparse dimension vectors which are not trivially sparse have the form $(k-2, k^3; 2k-1), ((k-2)^3, k; 2k-3)$,  $(1^4,3; 5)$, $(1^3, 3^2; 5)$, $(2, 4^4; 5)$ or $(2^2, 4^3; 5)$.

From now on, we may assume that $k \leq \frac{n}{2}$. By Theorem \ref{t:sl}, we may also assume that $a+b+c >4$. 

$\bullet$ (i)  If $a(k-2) + b(k-1) + ck \leq n+1$, then the vector is dense by Lemma \ref{l:4.5}. 

$\bullet$ (ii)  If $a (k-2) + b (k-1) + (c-1) k = n-r$ with $0 < r < k$, then by $c$ applications of Lemma \ref{l:3}, the density of $\ul{\bf{d}}$ reduces to the density of $((k-r)^c, (k-2)^a, (k-1)^b; n-cr)$.  If $r=1, 2$, then we are reduced to the equidimensional or the two-step case. Otherwise, $a+b$ additional applications of Lemma \ref{l:3}, reduces the density to the equidimensional case $((k-r)^{a+b+c}; n-cr-b(r-1)-a(r-2))$. In either case, a straightforward calculation shows that the vector is not trivially sparse. Therefore, the vector is dense by the cases $a=b=0$ or $a=0$ of the theorem. 

$\bullet$ (iii)  If  $a (k-2) + b (k-1) + (c-1) k = n$, by $(c-1)$ applications of Lemma \ref{l:6}, we may assume that $c=1$. By Lemma \ref{l:7}, we reduce the density of $\ul{\bf{d}}$ to the density of $((k-2)^a, (k-1)^b; k)$. The latter vector is trivially sparse if and only if $\ul{\bf{d}}$ is trivially sparse.  After taking the complement, the only vector of this form which is not dense is $(1^2, 2^2; 3)$ by \S \ref{sec-size}. We thus obtain a family of sparse vectors of the form $(1^2, 2^2, 3^c; 3c+3)$ and the complement $((3c)^c, (3c+1)^2, (3c+2)^2; 3c+3)$.

$\bullet$ (iv) Finally, by Lemma \ref{l:nl}, we have one of the following three cases:
\begin{enumerate}
\item $c\geq 2$, $a(k-2)+b(k-1) +(c-2)k = n+r$ for some $0<r<k$.
\item $c=1$, $b\geq 1$, $a(k-2)+(b-1)(k-1) = n+r$ for some $0 < r<k-1$.
\item $c=1$, $b=0$, $(a-1)(k-2) = n+r$ for some $0<r<k-2$.
\end{enumerate}
By Lemma \ref{l:9} and taking the complement, the density of $\ul{\bf{d}}$ is equivalent to the density of $\ul{\bf{d}}' = ((k'-2)^c, (k'-1)^b, (k')^a; n')$, where 
\begin{enumerate}
\item $\ul{\bf{d}}'=((k-r)^c, (k-r+1)^b, (k-r+2)^a; 2k-r),$
\item $\ul{\bf{d}}'=(k-r-1, (k-r)^b, (k-r+1)^a; 2k-r-1),$
\item  $\ul{\bf{d}}'= (k-2-r, (k-r)^a; 2k-r-2),$ respectively.
\end{enumerate}
An easy calculation shows that these vectors are trivially sparse if and only if the original vectors are trivially sparse. Recall we are assuming that $a+b+c> 4$. If $2k' > n'$, by the analysis above, $\ul{\bf{d}}' $ is one of $(1^3, 3^2; 4)$,  $(1^4, 3; 5)$, $(1^3, 3^2; 5)$ up to taking complements. Otherwise, we may assume that $2k' \leq n'$ and apply our inductive hypotheses. Note that $n' < n$. If  $\ul{\bf{d}}' $ is in case (i), (ii) or (iii), then it is dense except if it is of the form $(1^2, 2^2, 3^c; 3c+3)$. If $\ul{\bf{d}}' $ is in case (iv), we can apply Lemma \ref{l:9} to reduce the ambient dimension except when  $\ul{\bf{d}}' $ is one of the vectors $(1^3, 3^2; 4)$,  $(1^4, 3; 5)$, $(1^3, 3^2; 5)$ or their complements. Hence, any vector which is not trivially sparse and fails to be dense can be reduced to one of $(1^2, 2^2, 3^c; 3c+3)$, $(1^3, 3^2; 4)$,  $(1^4, 3; 5)$, $(1^3, 3^2; 5)$, by repeatedly applying Lemma \ref{l:9} and taking complements.   This concludes the proof of the theorem.
\end{proof}

\begin{cor}
Let $a_1 \leq \cdots \leq a_k$ be positive integers. Let $r$ be a non-negative integer such that $a_k -a_1 -1 \leq r$.  Let $n = \sum_{i=1}^{k-1} a_i + r$. Then the dimension vector  $\ul{\bf{d}}= (a_1, \dots, a_k; n)$ is dense. 
\end{cor}

\begin{proof}
If $r\geq a_k$, then $\ul{\bf{d}}$ is dense by Lemma \ref{l:2}.  We may assume $r< a_k$. If $a_k = a_1$ or $a_k = a_1 +1$, then the vector is dense by Theorem \ref{t:ed}. Otherwise, $r>0$. If, in addition, $r \geq a_k -a_1$,  $k$ applications of Lemma \ref{l:3} reduces the density of $\ul{\bf{d}}$ to the density of  the vector $((a_k-r)^k; (k-1)(a_k-r))$. Let $l$ be the number of integers in the sequence equal to $a_1$. If  $r = a_k -a_1-1>0$, then $k-l$ applications of Lemma \ref{l:3} reduces the density of $\ul{\bf{d}}$ to the density of  $(a_1^l, (a_1+1)^{k-l}; (k-1)a_1 + k-l-1)$. Both are dense by  Theorem \ref{t:ed}.
\end{proof}


\section{Further examples and questions}\label{s:fq}

In this section, we use our  techniques to produce examples of dense dimension vectors of large length and size. Lemma \ref{l:6} and Lemma \ref{l:7} proved in Section \ref{s:rt} have the special feature of reducing the length and the ambient dimension at the same time. This allows us to employ these lemmas in reverse  and produce sequences of dense dimension vectors with large length and size out of a given dense dimension vector.

Recall that the Fibonacci sequence $\{ F_i \}_{i \geq 0}$ is the sequence of nonnegative integers recursively defined by acquiring $F_0 = 0$, $F_1 = 1$, and $F_{i+2} = F_{i+1} + F_i$ for all $0 \leq i$.

\begin{prop}
Let $(a_1, \dots, a_r; b)$ be a dense dimension vector such that for any $1 \leq t \leq r$, $b + a_t \leq n := \sum_{i=1}^r a_i$. Then for every $k \geq 0$, the dimension vector
$$\ul{\bf{d}}_k := (a_1, \dots, a_r, b, F_1.n+F_0.b, F_2.n+F_1.b, \dots, F_k.n + F_{k-1}.b; F_{k+1}.n + F_k.b),$$
where $\{ F_i \}_{0 \leq i}$ is the Fibonacci sequence, is dense.
\end{prop}

\begin{proof}
The proof is by induction on $k$. For the base case of the induction, we need to show that the dimension vector $\ul{\bf{d}}_0 = (a_1, \dots, a_r, b; n)$ is dense, which follows from our assumptions and Lemma \ref{l:7}. Assume that $\ul{\bf{d}}_k$ is dense. Lemma \ref{l:7} implies that the dimension vector
$$(a_1, \dots, a_r, b, F_1.n+F_0.b, \dots, F_{k+1}.n + F_k.b; (1 + \sum_{i=0}^k F_i).n + (1 + \sum_{i=0}^{k-1} F_i).b)$$
is dense. Hence, it suffices to show that
\begin{equation} \label{e:f}
F_{t+2} = 1 + \sum_{i=0}^t F_i
\end{equation}
for any $t \geq 0$. We verify the identity \eqref{e:f} by induction on $t$. The case $t=0$ is evident as $F_2 = 1 = 1 + F_0$. Assuming \eqref{e:f}, we can compute
$$F_{t+3} = F_{t+2} + F_{t+1} = 1 + \sum_{i=0}^t F_i + F_{t+1} = 1 + \sum_{i=0}^{t+1} F_i.$$ This concludes the induction and the proof of the proposition.
\end{proof}

\begin{ex}
If we apply the above proposition to the dense dimension vector $(1,1,1; 2)$, with $n=3$ and $b=2$, we get the sequence
$$\ul{\bf{d}}_k = (1,1,1,2, 3F_1 + 2F_0, 3F_2 + 2F_1, \dots, 3F_k + 2F_{k-1}; 3F_{k+1} + 2F_k)$$
of dense dimension vectors. On the other hand, one can easily check by induction on $k$ that $3F_k + 2F_{k-1} = F_{k+3}$ (in fact, more generally, $F_t F_k + F_{t-1} F_{k-1} = F_{t+k-1}$ for all natural numbers $t$ and $k$). Therefore, we obtain that for any $k \geq 0$ the dimension vector
$$\ul{\bf{d}}_k = (1, F_1, F_2, \dots, F_{k+3}; F_{k+4})$$
is dense. Similarly, by applying the above proposition to the dense dimension vectors that we found in previous sections, we can construct infinitely many sequences of ``Fibonacci type" that are dense.
\end{ex}

\begin{prop}
Let $(a_1, \dots, a_r, b; n)$ be a dense dimension vector with $n = \sum_{i=1}^r a_i$. Then, for every $k \geq 1$, the dimension vector
$$\ul{\bf{d}}_k = (a_1, \dots, a_r, b^k; n+ (k-1)b)$$
is dense.
\end{prop}

\begin{proof}
We prove the assertion by induction on $k$. The induction basis $k = 1$ is our hypothesis. If we assume that the dimension vector $\ul{\bf{d}}_k$ is dense, Lemma \ref{l:6} implies that the dimension vector $\ul{\bf{d}}_{k+1}$ is dense, too.
\end{proof}

The reduction lemmas can be applied to study the density of arbitrary dimension vectors with size bounded by $n/2$. Unfortunately, after the reduction the dimension of some of the vector spaces might be more than half the ambient dimension. The same techniques can be applied even when some of the dimensions are greater than $n/2$. However,  after the reduction, the new problem is no longer the density of the $\pgl{n}$ action on a product of Grassmannians, but the density of the action on a subvariety of a product of flag varieties. This was already encountered in the proof of Lemma \ref{l:9}.  The following example is typical.

\begin{ex}
We show that the dimension vector $(5,5,5,5,13;14)$ is not dense. The reader can check that the same argument works for dimension vectors of the form $(k,k,k,k,3k-2; 3k-1)$. Let $U_1, U_2, U_3, U_4, W$ be general linear subspaces of a $14$-dimensional vector space $X$, where $\dim U_i = 5$ for $1 \leq i \leq 4$ and $\dim W= 13$. We first reduce checking the density to checking the density of a configuration where the ambient vector space has dimension 10.

Let $V= U_1 + U_2$. Let $U_i'= U_i \cap V$ for $i=3,4$ and let $W_1 = W \cap V$. Finally, let $$W_2 = ((U_3 \cap W) + (U_4 \cap W)) \cap V.$$ Then there is a natural restriction morphism $$f: \stab(U_1, U_2, U_3, U_4,W; X) \rightarrow \stab(U_1, U_2, U_3', U_4', W_1, W_2; V).$$ The new twist in this case is that $W_2 \subset W_1$, hence the new data is not a point of a product of Grassmannians, but of a product of partial flag varieties. Since the expected dimensions of the stabilizers are equal and both configurations can be taken to be generic, the density of one configuration is equivalent to the density of the other configuration by Lemma \ref{l:ml}.

Next we reduce the problem to one where the ambient dimension is 6. Set $V' = U_3'+U_4'+W_2$. Let $U_i' = U_i \cap V'$ for $i=1,2$, let $W_1' = W_1 \cap V'$ and $$W_3 = ((U_1 \cap W_1) + (U_2 \cap W_1)) \cap V'.$$ There is a natural restriction morphism
$$f: \stab(U_1, U_2, U_3', U_4', W_1, W_2; V) \rightarrow \stab(U_1', U_2', U_3', U_4', W_1', W_2, W_3; V').$$ Again both configurations are generic subject to the restriction $W_2 \subset W_1$ and $W_2, W_3 \subset W_1'$ and by Lemma \ref{l:ml} the density of one is equivalent to the density of the other. The new twist at this stage is that both $W_2, W_3 \subset W_1'$. Hence, this is a configuration parameterized by a subvariety of a product of Grassmannians and flag varieties defined by imposing some linear conditions on the vector spaces. At this stage, it is clear that the configuration is not dense since the configuration $(1,1,4,4;5)$ exists as a subconfiguration by taking $T_1 = (U_1'+U_3')\cap W_1'$, $T_2 = (U_2' + U_4') \cap W_1'$.
\end{ex}

We can speculate that whenever a dimension vector is not dense, there is always a configuration of vector spaces obtained by repeatedly taking spans, intersections and projections that gives a configuration which is trivially sparse.  Based on the previous example and to get a better inductive set up, the following generalization of Problem \ref{q:main} may be more natural.

\begin{pro}
Let $X$ be a subvariety of a product of flag varieties $\prod_{i=1}^k F(d_{i,1}, \dots, d_{i, j_i}; n)$ obtained by imposing linear relations on  the vector spaces. When does the diagonal action of $\pgl{n}$ have a dense orbit on $X$?
\end{pro}


\bibliographystyle{plain}

\end{document}